\documentclass[12pt,letter]{amsart}

\title
[Deformation of smooth DM stack via DGLA]
{Deformation of a smooth Deligne-Mumford stack via differential graded Lie algebra}
\author[Y. Nagai]{Yasunari Nagai}
\author[F. Sato]{Fumitoshi Sato}

\address{
Korea Institute for Advanced Study (KIAS),
207-43 Cheongnyangni 2-dong, Dongdaemun-gu, Seoul 130-722, Korea}
\curraddr{Institut f\"ur Mathematik, Universit\"at Mainz, 
Staudingerweg 9, 55099 Mainz Germany}
\email{nagai@mathematik.uni-mainz.de}

\address{
Korea Institute for Advanced Study (KIAS),
207-43 Cheongnyangni 2-dong, Dongdaemun-gu, Seoul 130-722, Korea}
\curraddr{Graduate School of Mathematics, 
Furocho, Chikusaku,Nagoya, 
464-8602 Japan}
\email{fumi@math.nagoya-u.ac.jp}

\date{15 July 2008: Revised version}

\usepackage{amsthm}
\usepackage{amssymb}
\usepackage{amsxtra}
\usepackage{enumerate}
\usepackage[all]{xy}

\usepackage{times,txfonts}
\usepackage{mathrsfs}

\theoremstyle{plain}
\newtheorem{theorem}[subsection]{Theorem}

\newtheorem{proposition}[subsection]{Proposition}
\newtheorem{propd}[subsection]{Proposition--Definition}
\newtheorem{corollary}[subsection]{Corollary}

\newtheorem*{theorem*}{Theorem}
\newtheorem*{corollary*}{Corollary}
\newtheorem*{MainTheorem}{Main Theorem}
\theoremstyle{definition}

\newtheorem{example}[subsection]{Example}

\theoremstyle{remark}

\newtheorem*{remark}{Remark}

\def\lto{\longrightarrow}

\def\CC{\mathbb C}

\DeclareMathOperator{\id}{id}

\DeclareMathOperator{\ob}{ob}

\DeclareMathOperator{\Spec}{Spec}
\DeclareMathOperator{\Ext}{Ext}

\DeclareMathOperator{\Ker}{Ker}

\DeclareMathOperator{\Def}{Def}
\def\rd{\partial}

\def\mathbi#1{\textbf{\em #1}}

\begin{document}

\baselineskip 14.58pt
\parskip 5pt

\maketitle
\vspace{-1cm}

\begin{abstract}
 For a smooth Deligne-Mumford stack over $\CC$, 
 we define its associated Kodaira-Spencer differential graded Lie
 algebra and show that 
 the deformation functor of the stack is isomorphic to 
 the deformation functor of the Kodaira-Spencer algebra
 if the stack is proper over $\CC$. 
\end{abstract}

\section*{Introduction}

Grothendieck and his followers established a general method to 
deal with the \emph{deformation theory}, which was initiated 
by Kodaira and Spencer. Grothendieck's method can be applied 
to the deformation of almost every algebro-geometric 
(or analytico-geometric) object. 

Recently, many people believe that a deformation theory over a field of
characteristic 0 should be `controlled' by a \emph{differential graded Lie
algebra} (DGLA in short). This principle seems to have come
from the researches concerning homotopy theory, quantization,
mirror symmetry etc (see, for example, \cite{K}).

One prototype example to this principle is the deformation theory of
compact complex manifold via Maurer-Cartan equation on the 
vector field valued $(0,1)$-forms. This is the 
\emph{Newlander-Nirenberg theorem} (or rather Kuranishi's proof of
the existence of Kuranishi space). 
If we restrict to infinitesimal deformations, we can describe 
the situation as a bijection between
\begin{equation}\label{NNcorresp}
\frac
{\left\{\mbox{Maurer-Cartan solutions in }KS_X^1\otimes m_A\right\}}
{\mbox{gauge equivalence}}
\cong 
\frac
{\left\{\begin{matrix}\mbox{deformations of a compact} \\
	\mbox{complex manifold $X$ over $A$}\end{matrix} \right\}}
{\mbox{isomorphisms}}
\end{equation}
where $A$ is a local artinian $\mathbb C$-algebra 
and $KS_X^{\bullet}=(A^{0,\bullet}_X(\Theta _X), \bar{\rd}, [-,-])$
the \emph{Kodaira-Spencer algebra} on $X$ (see Theorem \ref{classical}). 
This isomorphism is functorial in $A$. 
The left hand side is the deformation functor associated to the Kodaira-Spencer 
DGLA $KS_X^{\bullet}$, 
denoted by $\Def _{KS_X}$, and 
the right hand side is the usual deformation functor $\Def _X$ of $X$. 

Although the correspondence \eqref{NNcorresp} was originally based on highly 
analytic arguments by Newlander-Nirenberg, 
the statement itself concerns only infinitesimal deformations, 
therefore it is algebraic in nature. 
Actually, Iacono \cite{Ia} recently gave a purely algebraic proof 
(i.e., it involves no analysis of differential equations). 

With a view toward this situation, it is quite reasonable to expect that 
the correspondence \eqref{NNcorresp} can be generalized to the case of 
smooth Deligne-Mumford stacks. We can get some flavor from the case 
where $\mathscr X$ is given by a global quotient $[X/G]$ of a proper 
smooth algebraic variety $X$ by an action of a finite group $G$. 
Giving a deformation of $\mathscr X=[X/G]$ should be equivalent to giving 
a deformation of $X$ on which the $G$-action lifts. Therefore, the deformations of 
$\mathscr X$ are given by the $G$-invariant part $\Def _X(A)^G$ 
of the deformations of $X$. Since the correspondence \eqref{NNcorresp} is 
$G$-equivariant, if we take $(KS _X^{\bullet})^G$ as the DGLA, 
we get a functorial bijection
\[
\Def _{(KS _X)^G}(A)\overset{\sim}{\lto}
\Def _{X}(A)^G, 
\]
which describes the infinitesimal deformations of the stack 
$\mathscr X=[X/G]$ via a DGLA $(KS _X^{\bullet})^G$. 

In this article, we prove the following theorem:

\begin{MainTheorem}[Theorem \ref{main}]
Let $\mathscr X$ be a smooth separated analytic Deligne-Mumford stack. 
Then we can associate the Kodaira-Spencer differential graded Lie algebra 
$KS_{\mathscr X}^{\bullet}$ and there is a natural isomorphism of deformation functors
\[
\Gamma :\Def _{KS_{\mathscr X}}\to \Def _{\mathscr X}.
\]
\end{MainTheorem}

By a standard GAGA type argument (Proposition \ref{gaga2}), 
we also have a corresponding statement for a proper smooth (algebraic) 
Deligne-Mumford stack over $\mathbb C$. 

\begin{corollary*}
Let $X$ be a proper smooth Deligne-Mumford stack over $\mathbb C$. 
Then we also have the isomorphism 
$\Def _{KS_{\mathscr X}}\overset{\sim}{\to} \Def _{\mathscr X}$, 
where the right hand side is the deformation functor of 
algebraic deformations of $\mathscr X$.
\end{corollary*}

One obvious application of our main theorem is the deformation of 
a proper algebraic variety with only isolated quotient singularities, because 
the isolated quotient singularity is rigid if the dimension is not less than 
three \cite{Sch2}. 

\begin{corollary*}
Let $X$ be a proper algebraic variety with only isolated quotient singularities 
over $\mathbb C$ of dimension not less than three. 
Then the deformation functor of $X$ is isomorphic to 
the deformation functor of the Kodaira-Spencer algebra on the 
canonical covering stack $\mathscr X\to X$. 
\end{corollary*}

More generally, our main theorem describes the locus of 
equisingular deformations in the deformation space of a proper normal variety 
with only quotient singularities, i.e., a complex V-manifold. 

Our proof of the main theorem is parallel to the proof in \cite{Ia}. 
We mention some reasons why one can transplant the proof 
to the case of DM stacks. One reason is, of course, the algebraic nature 
of the arguments of the proof in \cite{Ia}. Another is the recent development 
of the deformation theory on algebraic stacks due to Aoki and Olsson 
\cite{A, O1}. In particular, it is crucial for our argument that 
Aoki's theorem on the equivalence of 
the deformation functor of an algebraic stack and the deformation functor of 
the simplicial space associated to the stack. 

The article goes as follows. In \S1, we review some general results 
on deformation functors and DGLA fixing our notation. The next section 
treats the deformation theory of stacks. We review Aoki's result and 
deduce some consequences in the case of DM stacks. We also treat 
GAGA type arguments for DM stacks as far as we need. 
In \S 3, we define the Kodaira-Spencer algebra on a smooth DM stack 
and prove a Dolbeault type theorem, which is essentially due to Behrend 
\cite{B}. After these preparations, we verify that the proof in \cite{Ia} works 
completely in the case of our Main Theorem. 

\paragraph{\bf Acknowledgement} 
This project was started at Mittag-Leffler Institute (Djursholm, Sweden), 
during the second author's participation in the program ``Moduli Spaces''. 
The second author would like to thank D. Iacono and M. Manetti for 
many valuable discussions. 
The authors are grateful to M. Aoki for correspondences which helped 
us in understanding the deformation theory of algebraic stacks. 
They also thank the referee for careful reading and comments. 

\section{Deformation theory and differential graded Lie algebras}

In this section, we review the theory of deformation functors of 
differential graded Lie algebras and its application to the deformation
theory of a compact complex manifold. The references are \cite{M} and
\cite{Ia}, Chap I. For \S1.3, see also \cite{F-M}. 

\subsection{}
Let $\mbox{\bf Art}$ be the category of local artinian $\mathbb C$-algebra $A$
such that $A/m_A\cong \mathbb C$, where $m_A$ is the maximal ideal of
$A$. We mean by a \emph{functor of artinian rings} a covariant functor 
\[
 D:\mbox{\bf Art}\to \mbox{\bf Set}
\]
such that $D(\mathbb C)$ is the one-point set. The \emph{tangent space} 
$t_D$ to a functor of artininan rings $D$ is defined by
\[
 t_D=D(\mathbb C[\varepsilon]), 
\]
where $\mathbb C[\varepsilon]$ is the ring of dual numbers $\mathbb
C[x]/(x^2)$. 

\subsection{}\label{gendef}
 Let $A,B,C$ be local artinian $\mathbb C$-algebras and
 \[
 \eta : D(B\times _A C)\to D(B)\times _{D(A)} D(C)
 \]
 be the natural map. We call a functor of artininan rings $D$
 a \emph{deformation functor} if it satisfies (i) if $B\to A$ is surjective, so is $\eta$, and 
 (ii) If $A=\mathbb C$, $\eta$ is bijective (\cite{M}, Definition 2.5).
 We remark that these conditions are closely related to Schlessinger's criterion 
 of existence of a \emph{hull} (see Remark to Definition 2.7 in \cite{F-M}).

\subsection{}\label{genobs}
Most deformation functors are described (in an
implicit way) by obstruction classes to the existence of a lifting of
a small extension and the space which parametrizes
the isomorphism classes of liftings in case the obstruction
vanishes.  Fantechi and Manetti abstracted 
these ``obstruction theories'' in the framework of functors of artininan rings 
(\cite{F-M} Definition 3.1, see also \cite{M} Defnition 2.12). 
 An \emph{obstruction theory} of a functor of artinian rings $D$ is
 a pair $(V,\ob(-))$ consisting of a $\mathbb C$-vector space $V$, the
 \emph{obstruction space}, and 
 a map $\ob(\alpha):D(\bar A)\to V\otimes I$, the \emph{obstruction
 map}, for every small extension 
 \[
  \alpha :\quad 0\lto I\lto A\lto \bar A\lto 0, 
 \]
 i.e., an extension with $I\cdot m_A=0$, satisfying the following conditions:
 \begin{enumerate}[(i)]
  \item If $\bar x\in D(\bar A)$ lifts to $D(A)$, $\ob(\alpha)(\bar
	x)=0$. 
  \item For any morphism $\varphi$ of small extensions
	\[
	 \xymatrix{\alpha _1: &
	0\ar[r] & I_1 \ar[r]\ar[d]_{\varphi _I} & A_1\ar[r]\ar[d]_{\varphi} 
	& \bar A_1 \ar[r]\ar[d]_{\bar\varphi} & 0\\
	\alpha _2: & 0\ar[r] & I_2 \ar[r] & A_2 \ar[r] & \bar A_2 \ar[r] & 0
	}\lower45pt\hbox{,}
	\]
	we have the compatibility
	$ \ob(\alpha _2)(\bar\varphi_*(\bar x))=(\id _V\otimes \varphi
	_I)(\ob(\alpha _1)(\bar x))$ for every $\bar x\in D(\bar A_1)$. 
 \end{enumerate}
 Moreover, if $\ob(\alpha)(\bar x)=0$ implies the existence of a lifting
 of $\bar x$ to $D(A)$, the obstruction theory is called
 \emph{complete}. 

\begin{proposition}[\cite{M}, Proposition 2.17]\label{criteria}
 Let $D_1$ and $D_2$ be deformation functors and $\varphi :D_1\to D_2$ a
 morphism of functors, $(V_1,\ob _{D_1})$ and $(V_2, \ob _{D_2})$
 obstruction theories for $D_1$ and $D_2$, respectively. Assume that
 \begin{enumerate}[(i)]
  \item $\varphi $ induces a surjection (resp. bijection) on the tangent
	spaces $t_{D_1}\to t_{D_2}$. 
  \item There is an injective linear map between obstruction spaces $\psi
	:V_1\to V_2$ such that $\ob _{D_2}\circ \varphi = \psi \circ \ob
	_{D_1}$. 
  \item The obstruction theory $(V_1, \ob _{D_1})$ is complete. 
 \end{enumerate}
 Then, the morphism $\varphi$ is smooth (resp. \'etale). 
\end{proposition}

\subsection{}
Now we recall the definition of a differential graded Lie algebra.
Let $L^{\bullet}=\bigoplus _{i\in \mathbb Z} L_i$ be a graded $\mathbb C$-vector space. 
A triple $(L^{\bullet}, d, [-,-])$ is a
 \emph{differential graded Lie algebra} (DGLA in short) if it satisfies 
  (i) $d=\sum d_i$ is a homogeneous differential of degree 1, 
  i.e., $d_i:L^i \to L^{i+1} $ and $d\circ d=0$, 
  (ii) the bracket $[-,-]$ is a homogeneous, graded skew-symmetric
	bilinear form on $L^{\bullet}$, i.e., $[L^i,L^j]\subset
	L^{i+j}$ and $[a,b]+(-1)^{\deg a \cdot \deg b}[b,a]=0$, 
  (iii) the graded Jacobi identity  
	$[a,[b,c]]=[[a,b],c]+(-1)^{\deg a\cdot \deg b}[b,[a,c]]$
	holds for the bracket, and 
  (iv) the graded Leibniz rule $d[a,b]=[da,b]+(-1)^{\deg a}[a,db]$ holds. 

\subsection{}
Given a DGLA $L^{\bullet}$, we can associate its Maurer-Cartan functor
as follow. Let $L^{\bullet}$ be a DGLA and $A$ a local artinian $\mathbb C$-algebra. 
We define the \emph{Maurer-Cartan functor} 
$MC_L: \mbox{\bf Art} \to \mbox{\bf Set}$ associated to $L^{\bullet}$ by 
\[
 MC_L(A)=\{x\in L^1\otimes m_A\mid dx+\frac 12[x,x]=0\},
\]
 where $d$ and $[-,-]$ is the DGLA structure  
 on $L^{\bullet}\otimes m_A$ induced 
 in an obvious way by $L^{\bullet}$. We call an element of $MC_L(A)$ a
 \emph{Maurer-Cartan solution}. 

\subsection{}\label{dgladef}
 In some cases, the space of Maurer-Cartan solutions $MC_{L}(A)$ is `too big'. 
 In such a case, we get a more reasonable deformation functor considering 
 the gauge action. 
 Let $L^{\bullet}$ be a DGLA and $a\in L^0\otimes m_A$ where 
 $A$ is a local artinian $\mathbb C$-algebra. 
 For $x\in L^1\otimes m_A$, the \emph{gauge action} of $a$ is given by
 \[
 e^a * x = x + \sum _{n=0}^{\infty} 
 \frac{[a,-]^n}{(n+1)!}([a,x]-da), 
 \]
 where $[a,-]^n(y)$ is the operator $[a,-]$ applied to $y$ $n$-times
 reccursively: \linebreak
 $[a,[a,[\cdots ,[a,y]\cdots]]]$. 
 Note that the gauge action preserves the space of the Maurer-Cartan 
 solutions $MC_L (A)$. Two Maurer-Cartan solutions 
 $x,y\in MC_L(A)$ are said to be \emph{gauge equivalent} 
 if there exists $a\in L^0\otimes m_A$ such that $y=e^a*x$. 
 We define the \emph{deformation functor} $\Def_L: \mbox{\bf Art} \to \mbox{\bf Set}$
 of a DGLA $L^{\bullet}$ by $\Def_L(A)=MC_L(A)/(\mbox{gauge equivalence})$. 
 One can easily check the following facts (see, for example, \cite{Ia} \S I.3.5): 
\begin{enumerate}[(i)]
 \item  The functor $\Def _L$ is a deformation functor in the sense of
 Definition \ref{gendef}. 
 \item  The tangent space to the deformation functor $\Def _L$ is
 canonically isomorphic to $H^1(L^{\bullet}, d)$. 
 \item There is a natural complete obstruction theory of $\Def _L$ 
       with the obstruction space $H^2(L^{\bullet}, d)$. 
\end{enumerate}

\subsection{}
One of the DGLA's which appear naturally 
in a geometric context is the Kodaira-Spencer algebra. 
Let $X$ be a complex manifold and $KS_X^p=A^{0,p}_X(\Theta _X)$ 
be the space of $C^{\infty}$-differential forms of type $(0,p)$ 
with (holomorphic) vector field coefficients. Then,
$(KS_X^{\bullet},\bar{\rd},[-,-])$ is a DGLA in a natural way, where
$[-,-]$ is a bracket induced by the Lie bracket on $\Theta _X$. We
call this DGLA the \emph{Kodaira-Spencer algebra}. 

\subsection{}\label{defsch}
We can reformulate the classical deformation theory of a compact complex
manifold by Kodaira-Spencer,
Newlander-Nirenberg and Kuranishi, using the framework of DGLA and its
deformation functor for infinitesimal deformations. 

Let $X$ be a scheme (resp. an analytic space) 
and $A$ be a local artinian $\mathbb C$-algebra. A \emph{lifting} of $X$ to $A$ is
a pair $(\widetilde X, \varphi)$ where $\widetilde X$ 
is a scheme (resp. an analytic space) flat over $A$ and $\varphi$ is 
an isomorphism $\widetilde X\times _{\Spec A} \Spec \mathbb C\overset{\sim}\to X$. 
Two liftings $(\widetilde X_1, \varphi _1)$ and $(\widetilde X_2, \varphi_2)$ are isomorphic 
if there exists an $A$-isomorphism $\tilde f:\widetilde X_1\to \widetilde X_2$ that is 
compatible with the marking isomorphisms, i.e. 
$\varphi _1=\varphi _2\circ (\tilde f\otimes _A \mathbb C)$. 
We define the \emph{deformation functor  
$\Def _{X} : \mbox{\bf Art} \to \mbox{\bf Set}$ of $X$}
by $\Def _X(A)=\{\mbox{isomorphism class of 
liftings $\widetilde X$ of $X$ to $A$}\}$. 

\begin{theorem}[See \cite{Ia} Theorem II.7.3]\label{classical}
Let $X$ be a complex manifold and $KS_X^\bullet$ 
the associated Kodaira-Spencer algebra. 
Then we have an isomorphism $\gamma : \Def _{KS_X} \to \Def _X$ between 
the deformation functors. 
\end{theorem}

\section{Deformation of a Deligne-Mumford stack}

Recently, Aoki \cite{A} explored the deformation theory of an algebraic
stack (in the sense of Artin), whose work depends on the preceeding work
by Olsson \cite{O1, O2}. In the first half of this section, 
we review some results from \cite{A}. 

We can define a deformation functor $\Def _{\mathscr X}$ 
of an algebraic stack $\mathscr X$ 
just as in (\ref{defsch}) 
for an algebraic stack $\mathscr X$, 
namely, for a local artinian $\mathbb C$-algebra $A$, 
we define $\Def _{\mathscr X}(A)$ to be the set of isomorphism 
classes of liftings $\widetilde{\mathscr X}$ of $\mathscr X$ 
to $A$
(here we note that we will only consider 1-isomorphism classes of 
liftings, so that our deformation functor is coarser than the one in 
\cite{A}, Definition 1.1).

We will write groupoid space for internal groupoid in the category of algebraic 
spaces (\cite{A}, Definition 2.1.2). 
To an algebraic stack $\mathscr X$, we can associate 
a choice of a smooth atlas $U\to \mathscr X$ and 
a groupoid space
\begin{equation}\label{gpspdiag}
\xymatrix{
T=R{}_s \underset{U}{\times} {}_t R \ar[r]^{\quad\;\;\; m} &
R \ar@(ur,ul)[]_i \ar@<2ex>[r]^s \ar[r]^t &
U \ar@<1ex>[l]^e
}
\end{equation}
such that $s$ and $t$ are smooth and $(s,t): R\to U\times U$ is 
quasi-compact and separated. Conversely, if we are given 
a groupoid space $R\rightrightarrows U$ with these properties, we can recover the 
algebraic stack $\mathscr X$. 

For a groupoid space $R\rightrightarrows U$, we can naturally define its associated 
deformation functor  $\Def _{(R\rightrightarrows U)} : \mbox{\bf Art} \to \mbox{\bf Set}$ 
in an obvious way. Aoki \cite{A} showed that the correspondence between algebraic stacks 
and groupoid spaces induces an isomorphism of these deformation functors. 

\begin{theorem}[\cite{A}, Proposition 3.2.5]\label{aoki}
Let $X$ be an algebraic stack and $R\rightrightarrows U$ be an associated groupoid space. 
Then there is a natural isomorphism of functors 
\[
C:\Def _{(R\rightrightarrows U)} \to \Def _{\mathscr X}.
\]
\end{theorem}

As a corollary, we get the following results. 

\begin{corollary}[\cite{A}, Corollary 3.2.6]\label{squarefree}
 Let $A\to\bar A$ be an extension with square zero ideal $I=\Ker
 (A\to \bar A)$ and $\mathscr X$ be an algebraic stack flat over
 $S=\Spec \bar A$. Denote by $p:X_{\bullet}\to S$ 
 the associated simplicial space over $S$ to a groupoid representation 
 $X_0=U\to \mathscr X$ . Then,  
  (i) There is an obstruction class 
	$\omega\in \Ext ^2(L_{X_{\bullet}/S}, p^* I)$ to the existence
	of a lifting of $\mathscr X$ to $A$, and 
  (ii) if $\omega$ vanishes, the set of isomorphism classes of liftings of
	$\mathscr X$ to $A$ is a torsor under 
	$\Ext ^1(L_{X_{\bullet}/S}, p^* I)$, 
 where $L_{X_{\bullet}/S}$ is the cotangent complex associated to the
 ringed \'etale topos on $X_{\bullet}$ \cite{Il} and $\Ext$ groups are
 also computed on the \'etale site.  
\end{corollary}

\begin{corollary}[\cite{A}, \S 4.2]
 The functor $\Def _{\mathscr X}$, and therefore $\Def _{(R\rightrightarrows
 U)}$, is a deformation functor in the sense of (\ref{gendef}). 
\end{corollary}

If we restrict to a Deligne-Mumford stack (DM stack in short), 
we can take an \'etale atlas $U\to \mathscr X$ so that all the projections
of the simplicial scheme $X_{\bullet}$ are \'etale. The transitivity
of the cotangent complex (II.2.1.5.6 in \cite{Il}) and the vanishing of
the cotangent complex for \'etale morphisms 
(Proposition III.3.1.1 in \cite{Il}) imply that $L_{X_{\bullet}/S}$
descends to the cotangent complex $L_{\mathscr X/S}$ on the \'etale site
of $\mathscr X$. This implies the following corollary:

\begin{corollary}\label{dmcase}
 For a DM stack $\mathscr X$, 
 the liftings of $f:\mathscr X\to S$ with square zero ideal 
 $I$ are controlled by the Ext groups on the \'etale site 
 $\Ext ^i(L_{\mathscr X/S},f^*I),\; (i=1,2)$. In
 particular, the deformation functor 
 $\Def _{\mathscr X}$ is isomorphic to the deformation functor of 
 the structure sheaf $\mathscr O_{\mathscr X}$, i.e., the functor 
 \[
 \Def _{\mathscr O_{\mathscr X}}(A)=
 \left\{\begin{matrix}
	\mbox{isomorphism class of a sheaf of algebras $\mathscr S$ flat over $A$}\\
  \mbox{on the \'etale site over $\mathscr X$ such that 
  $\mathscr S\otimes _A \mathbb C\cong \mathscr O_{\mathscr X}$}
  \end{matrix}\right\}. 
 \]
\end{corollary}

In the rest of this section, 
we review some comparison results of GAGA type as much as we need. 

In the definition of algebraic stacks, if we replace the category of schemes 
by the category of analytic spaces, 
we get the concept of 
\emph{analytic stacks} in Artin's sense, or Deligne--Mumford's sense.
Given an algebraic stack $\mathscr X$, we have the associated 
analytic stack $\mathscr X_{an}$; one way to see this is to take a
groupoid $R\rightrightarrows U$ representing $\mathscr X$ and take 
the associated groupoid of analytic spaces $R _{an}\rightrightarrows
U_{an}$, which induces $\mathscr X_{an}$. 
Similarly, we have the concept of liftings of an analytic stack to a local 
artinian ring and if we have a lifting $\widetilde {\mathscr X}$ of an algebraic stack 
$\mathscr X$ to $A$, we have the corresponding lifting of analytic stack 
$\widetilde {\mathscr X}_{an}$ of $\mathscr X_{an}$ to $A$, i.e., we have 
a natural transformation
\[
 \alpha :\Def _{\mathscr X}\to \Def _{\mathscr X_{an}}. 
\]
Note that the proof of Theorem \ref{aoki} holds true for
analytic stacks by the same proof as in \cite{A,Il,O1}; therefore,  
Corollary \ref{dmcase} also holds true for analytic DM stacks. 

\begin{proposition}\label{gaga1}
 Let $\mathscr X$ be a DM stack proper over an affine $\mathbb C$-scheme
 $S$ and $\mathscr F$ be a
 coherent sheaf on $\mathscr X$. Then we have a natural isomorphism
 \[
 H^p(\mathscr X, \mathscr F)\cong H^p(\mathscr X_{an},\mathscr F_{an}). 
 \]
\end{proposition}

\begin{proof}
 This is standard. Chow's Lemma for DM stacks (\cite{L-M}, Corollaire
 16.6.1) and the d\'evissage technique of Grothendieck (\cite{EGA3}, \S 3, or
 \cite{Gr}) reduces the proof to the classical case \cite{Se}. 
\end{proof}

\begin{proposition}\label{gaga2}
 The natural transformation $\alpha :\Def _{\mathscr X}\to \Def
 _{\mathscr X_{an}}$ is an isomorphism if $\mathscr X$ is a smooth DM
 stack proper over $\mathbb C$. 
\end{proposition}

\begin{proof}
 As we assumed that $\mathscr X$ is smooth, 
 the infinitesimal deformation space and 
 the obstruction space in Corollary \ref{dmcase} are
 given by  $H^i(\mathscr X,\Theta _{\mathscr X}),\; (i=1,2)$. 
 We also have the same statement for $\mathscr X_{an}$. 
 If we apply Proposition \ref{gaga1} for $\mathscr F=\Theta _{\mathscr
 X}$, we get our proposition. 
\end{proof}

We remark that as a corollary of the proposition, we get an isormorphism
of the deformation functors $\Def _{(R\rightrightarrows U)}\cong \Def
_{(R_{an}\rightrightarrows U_{an})}$. The obstruction theory in
the proof of the proposition above is in fact a complete obstruction
theory in the sense of Definition \ref{genobs}.

\begin{proposition}
 Let $\mathscr X$ be a smooth (algebraic or analytic) 
 DM stack and $\Def _{\mathscr X}$ its
 deformation functor. 
 There is a natural complete obstruction theory for $\Def _{\mathscr X}$
 with obstruction space $H^2(\mathscr X, \Theta _{\mathscr X})$. 
 If $\mathscr X$ is proper over $\mathbb C$, the obstruction theory is compatible 
 with the operation of taking the associated analytic stack.
\end{proposition}

\begin{proof}
 This is also implied by \cite{A}, Corollary 3.2.6. 
 The compatibility with base change (ii) in
 Definition \ref{genobs} goes back to the definition of the obstruction
 class in \cite{Il}, Th\'eor\`eme 2.1.7. 
\end{proof}

\section{Kodaira-Spencer algebra associated to a smooth DM stack}

The observation in the introduction suggests that the deformation theory of 
a smooth Deligne-Mumford stack 
should also be controlled by a DGLA something like 
the Kodaira-Spencer algebra, and such a DGLA should be realized 
as a DGL sub-algebra of the Kodaira-Spencer algebra of an atlas of the stack. 

For a Deligne-Mumford stack $\mathscr X$, 
the cotangent sheaf $\Omega _{\mathscr X}$ is 
a well-defined $\mathscr O_{\mathscr X}$-coherent sheaf and 
it is locally free if $\mathscr X$ is smooth. 
Therefore, the tangent sheaf $\Theta _{\mathscr X}$ is also a locally free 
sheaf for a smooth $\mathscr X$. Similarly, on the \'etale site over
$\mathscr X_{an}$, we have the sheaf of $C^{\infty}$-differentials 
$\mathscr A^{p,q}_{\mathscr X_{an}}$. 

In the rest of the article, we will work in the analytic category. 
Moreover, we always assume that 
\emph{a DM stack is locally compact and second-countable}, 
i.e., we assume that every analytic space 
in the \'etale site of a DM stack (in particular an atlas)
is locally compact and second-countable 

Now take an \'etale atlas $U\to \mathscr X$. 
Let $R=U\times _{\mathscr X} U$ be the ``space of relations'' and
$s, t: R\to U$ the first and second projections, respectively. 
Since $U$ is a smooth space, we have the associated Kodaira-Spencer algebra 
$KS_U^{\bullet}=A^{0,\bullet}_U(\Theta _U)$. 
The \'etale morphism $s:R\to U$ induces a map 
$s^*: A^{0,p}_U(\Theta _U)\to A^{0,p}_R(\Theta _R)$
and the same holds for $t$. 

\begin{propd}
Let $\mathscr X$ be a smooth DM stack. Define $KS_{\mathscr X}^{\bullet}$ by  
\[
KS_{\mathscr X}^p = \{ x\in A^{0,p}_U(\Theta _U) \mid s^*x=t^*x\}
 \subset A^{0,p}_U(\Theta _U).
\]
Then, 
\begin{enumerate}[(i)]
\item $KS_{\mathscr X}^p$ does not depend on the choice of an atlas 
        $U\to \mathscr X$. More precisely, 
        $KS_{\mathscr X}^p$ is the space of global sections 
        $\Gamma (\mathscr X,\mathscr A^{0,p}_{\mathscr X}(\Theta _{\mathscr X}))$. 
\item $KS^{\bullet}_{\mathscr X}=\bigoplus _p KS_{\mathscr X}^p$ is 
a differential graded Lie sub-algebra of $KS_U^{\bullet}$. 
\end{enumerate}
We call $KS_{\mathscr X}$ 
the \emph{Kodaira-Spencer algebra of the DM stack} $\mathscr X$.
\end{propd}

\begin{proof}
 (i) is nothing but the descent property of 
 $\mathscr A^{0,p}_{\mathscr X}(\Theta _{\mathscr X})$  
 on the analytic-\'etale site on $\mathscr X$. 
 For (ii), it is enough to show that $KS^{\bullet}_{\mathscr X}$ 
 is closed under $\bar{\rd}$ and the bracket $[-,-]$ of
 $KS^{\bullet}_U$. This is equivalent to saying that $\bar{\rd}$ and
 $[-,-]$ commute with $s^*$ and $t^*$. But this is obvious because $s$
 and $t$ are \'etale. 
\end{proof}

\begin{example}
For a global quotient DM stack $[X/G]$, we can take $X\to [X/G]$ as an atlas. 
The proposition immediately implies that $KS_{[X/G]}^{\bullet}=(KS_X^{\bullet})^G$. 
\end{example}

Given the Kodaira-Spencer algebra $KS^{\bullet}_{\mathscr X}$ for a
smooth DM stack, we can, of course, consider its deformation functor
$\Def _{KS _{\mathscr X}}$. Its tangent space and complete obstruction
space (\ref{dgladef}) can be computed by the following Dolbeault type theorem. 

\begin{theorem}\label{dolbeault}
 Let $\mathscr X$ be a smooth separated DM stack over $\mathbb C$. 
 Then there is an isomorphism 
 $H^p(KS ^{\bullet}_{\mathscr X}, \bar{\rd}) \overset{\sim}{\to} 
 H^p(\mathscr X,\Theta _{\mathscr X})$ 
 for all $p$. 
\end{theorem}

\begin{proof}
 Take an \'etale atlas $U\to \mathscr X$ with $U$ smooth and Stein 
 and consider the associated \v{C}ech--Dolbeault
 complex as usual:
 \begin{equation}\label{CechDolbeault}
 \xymatrix@=12pt{
 0 \ar[r] &\Gamma(U_p, \Theta )\ar[r]
 &\Gamma (U_p, \mathscr A^{0,0}(\Theta))\ar[r]
 & \cdots \ar[r] & \Gamma (U_p, A^{0,q}(\Theta))\ar[r]
 &\cdots \\
& \vdots \ar[u] & \vdots \ar[u] & & \vdots \ar[u] \\
 0 \ar[r] &\Gamma (U_1, \Theta )\ar[r]\ar[u]
 &\Gamma (U_1, \mathscr A^{0,0}(\Theta))\ar[r]\ar[u]
 & \cdots \ar[r] & \Gamma (U_1, A^{0,q}(\Theta))\ar[r]\ar[u]
 &\cdots \\
 0 \ar[r] &\Gamma (U_0, \Theta )\ar[r]\ar[u]_{t^*-s^*}
 &\Gamma (U_0, \mathscr A^{0,0}(\Theta))\ar[r]\ar[u]_{t^*-s^*}
 & \cdots \ar[r] & \Gamma (U_0, A^{0,q}(\Theta))\ar[r]\ar[u]_{t^*-s^*}
 &\cdots \\
 & 0 \ar[r]\ar[u] & KS^0_{\mathscr X}\ar[r]\ar[u]
 & \cdots \ar[r] & KS^q_{\mathscr X}\ar[r]\ar[u]\ar[r]
 & \cdots \\
 & & 0 \ar[u] & & 0\ar[u]
 }
 \end{equation}
 where $U_p$ is the $(p+1)$-fold self fiber product of $U$ over
 $\mathscr X$. The vanishing of cohomologies of coherent sheaves on a
 Stein space implies that the complex appeared in the leftmost column
 calculates $H^p(\mathscr X,\Theta _{\mathscr X})$. 
 The Dolbeault theorem on an usual complex manifold $U_p$
 implies that the rows but the bottom one is exact. On the other hand, 
 Behrend \cite{B} showed that if we replace $U$ by its `refinement', 
 we have a partition of unity associated with $U$ (\cite{B}, Definition
 22), and this immediately implies that the columns except the leftmost
 one are exact (\cite{B}, Proposition 23). A standard double complex
 argument leads to our theorem.
\end{proof}

\section{Infinitesimal Newlander-Nirenberg theorem for smooth DM stack}

In this section, we complete the proof of our Main Theorem. 
After the preparations in \S\S 2 and 3, 
our proof is an honest transplantation of the proof of Theorem
\ref{classical} found in \cite{Ia},
Chap. II to the context of smooth DM stacks. 
We begin with the following theorem. 

\begin{theorem}\label{NNmorph-p}
Let $\mathscr X$ be a
smooth separated analytic Deligne-Mumford stack and $KS_{\mathscr X}^\bullet$ 
be the associated Kodaira-Spencer algebra. Then, 
there exists a natural injective morphism between the deformation functors
\[
\Gamma :\Def _{KS_{\mathscr X}}\to \Def _{\mathscr X}
\]
in the analytic category. 
\end{theorem}

\begin{proposition}\label{NNmorph}
Let $A$ be a local artinian $\mathbb C$-algebra with the maximal
ideal $m_A$ and $x\in MC_{KS_{\mathscr X}}(A)$ be a Maurer-Cartan solution 
of $KS_{\mathscr X}$. Then, to $x$ we can associate a lifting 
$\widetilde{\mathscr X}\in \Def _{\mathscr X}(A)$.
\end{proposition}

\begin{proof}
 Corollary \ref{dmcase} implies that it is enough to construct a sheaf of
 algebras $\mathscr S_x$ on $\mathscr X$ flat over $A$ such that $\mathscr S_x\otimes
 _A \mathbb C\cong \mathscr O_{\mathscr X}$.  
 Take a smooth Stein space $U$ as an atlas $U\to \mathscr X$. 	
 The tangent space to the deformation functor $\Def _{KS_U}$ associated 
 to the Kodaira-Spencer algebra $KS_U^\bullet$ of $U$ is isomorphic to 
 $H^1(U, \Theta _U)$, which is in fact trivial, for $U$ is smooth and Stein. 
 This means that $\Def _{KS_U}$ is trivial (\cite{Ia} Lemma II.7.1). 
 In other words, for any Mauer-Cartan solution $x\in MC_{KS_{\mathscr X}}(A)
 \subset MC_{KS_U}(A)$, there is a 
 $C^{\infty}$-vector field $a\in A^{0,0}_U(\Theta _U)\otimes m_A$ such that $e^a * x=0$. 
 We define the operator $\mathbi l_x : \mathscr A^{0,0}_U\otimes A
 \to \mathscr A^{0,1}_U \otimes A$ 
 for $x=\sum _{i,j} x_{ij} d\bar{z}_i\frac{\rd}{\rd z_j}$ by the contraction
 \[
 \mathbi l_x(f) 
 = - \sum _{i,j} x_{ij} \frac{\rd f}{\rd z_j}d\bar{z}_i, 
 \]
 where $z_i$ are local holomorphic coordinates on $U$. 
 By an explicit calculation of the gauge action (\cite{Ia}, Lemma II.5.5), we have 
 \[
 e^a \circ (\bar{\rd} + \mathbi l_x ) \circ e^{-a} 
 = \bar{\rd} + e^a * \mathbi l_x =\bar{\rd},
 \]
 where the last equality follows from $e^a * x=0$. This means that 
 the diagram 
 \[
 \xymatrix{
 \mathscr S_x=\Ker (\bar{\rd} + \mathbi l_x) \ar[r] \ar[d]_{e^a} &
 \mathscr A^{0,0}_U\otimes A \ar[r]^{\bar{\rd}+\mathbi l_x} \ar[d]_{e^a} &
 \mathscr A^{0,1}_U\otimes A  \ar[d]_{e^a} \\
 \mathscr O_{\widetilde U}\ar[r] & 
 \mathscr A^{0,0}_U\otimes A \ar[r]^{\bar{\rd}} &
 \mathscr A^{0,1}_U\otimes A
 }\]
 is commutative, where $\widetilde U=U\times \Spec A$. 
 In particular, $e^a : \mathscr S_x \overset{\sim}{\to}
 \mathscr O_{\tilde U}$ is an isomorphism of sheaf of algebras over $A$. 
 Since $\mathscr O_{\widetilde U}$ is flat over $A$, $\mathscr S_x$ is 
 also flat over $A$. 

 Let $y=s^*x=t^*x\in A^{0,1}_R(\Theta)\otimes m_A$. We have an analogous diagram on 
 $R$:
 \[
 \xymatrix{
 s^{-1}\mathscr S_x=\Ker (\bar{\rd} + \mathbi l_{y}) \ar[r] \ar[d]_{e^{s^*a}} &
 \mathscr A^{0,0}_R\otimes A \ar[r]^{\bar{\rd}+\mathbi l_y} \ar[d]_{e^{s^*a}} &
 \mathscr A^{0,1}_R\otimes A  \ar[d]_{e^{s^*a}} \\
 s^{-1}\mathscr O_{\widetilde U}\cong \mathscr O_{\widetilde R}\ar[r] & 
 \mathscr A^{0,0}_R\otimes A \ar[r]^{\bar{\rd}} &
 \mathscr A^{0,1}_R\otimes A,  
 }\]
 where $\widetilde R=R\times \Spec A$, 
 and the same diagram also for $t^{-1} \mathscr S_x$. 
 Hence we have $s^{-1}\mathscr S_x=t^{-1}\mathscr S_x$ as sub-sheaves of 
 $\mathscr A^{0,0}_R\otimes A$. This means that the descent data for
 $\mathscr A^{0,0}\otimes A$ induces a descent data for $\mathscr S_x$. 
 Therefore, $\mathscr S_x$ descends to a sheaf of algebras on $\mathscr X$ flat over $A$. 
 $\mathscr S_x\otimes _A\mathbb C\cong \mathscr O_{\mathscr X}$ is obvious.  
\end{proof}

The proposition says that we have a morphism 
$\hat\Gamma :MC_{KS_{\mathscr X}}\to \Def _{\mathscr X}$. 
The following proposition concludes the proof of Theorem \ref{NNmorph-p}.

\begin{proposition} \label{modgauge}
$\hat\Gamma$ descends to an injective morphism 
\[
\Gamma :\Def _{KS_{\mathscr X}}\to \Def _{\mathscr X}.
\]
\end{proposition}

\begin{proof}
 Assume $x,y\in MC_{KS _\mathscr X}(A)$ are gauge equivalent, i.e., there
 exists $r\in KS^0_{\mathscr X}\otimes m_A$ such that $e^r*x=y$. By \cite{Ia},
 Lemma II.5.5, we have a commutative diagram
 \[
 \xymatrix{
 \mathscr S_x=\Ker (\bar{\rd} + \mathbi l_x) \ar[r] \ar[d]_{e^r} &
 \mathscr A^{0,0}_U\otimes A \ar[r]^{\bar{\rd}+\mathbi l_x} \ar[d]_{e^r} &
 \mathscr A^{0,1}_U\otimes A  \ar[d]_{e^r} \\
 \mathscr S_y=\Ker (\bar{\rd} + \mathbi l_y) \ar[r] & 
 \mathscr A^{0,0}_U\otimes A \ar[r]^{\bar{\rd} + \mathbi l_y} &
 \mathscr A^{0,1}_U\otimes A
 }\]
 whose columns are isomorphisms. 
 $r\in KS_{\mathscr X}^0\otimes m_A$ implies $e^{s^*r}=e^{t^*r}$ so that $e^r$ 
 descends to an isomorphism between $\mathscr S_x$ and $\mathscr S_y$ 
 on $\mathscr X$.  This means that $\hat\Gamma$ factors
 through $\Gamma$. 

 To prove the injectivity of $\Gamma$, we show that
 given an isomorphism $\psi :\mathscr S_x\to \mathscr S_y$ satisfying 
 $s^*\psi = t^*\psi$, 
 there exists $r\in KS_{\mathscr X}^0\otimes m_A$ such that $\psi = e^r$. 
 It is enough to show this for small extensions by inductive
 argument. In other words, we can assume that there is an extension
 \[
  0\lto I\lto A \lto \bar A\lto 0
 \]
 with $I\cdot m_A=0$ and $p=x-y\in A^{0,1}_U(\Theta _U)\otimes I$. 
 Under this assumption, we have 
 \[
 0= \bar\rd (y+p)+\frac 12[y+p, y+p]
 = \bar\rd y + \bar\rd p + \frac 12[y,y]=\bar \rd p. 
 \]
 Since we assumed that $U$ is Stein, $H^1(U,\Theta _U)$ vanishes. 
 By the (usual) Dolbeault theorem, this means that the Dolbeault complex
 $(A^{0,\bullet}_U(\Theta _U), \bar\rd)$ is exact at the degree 1 place. 
 Therefore, we have $u\in KS^0_{\mathscr X}\otimes I$ such that $\bar\rd u=p$. 
 Then, we have
 \[
 e^u * x = x + \sum _{n=0}^{\infty} \frac{[u,-]^n}{(n+1)!}([u,x]-\bar\rd u)
 = x - \bar\rd u = x - p = y.
 \]
 Take $a\in A^{0,0}_U(\Theta _U)\otimes m_A$ which induces an isomorphism
 $e^a:\mathscr S_x\to \mathscr O_{\widetilde U}$ and $\varphi $ be an
 automorphism of $\mathscr O_{\widetilde U}$ making the diagram 
 \[
 \xymatrix{
 \mathscr S_x \ar[r]^{\psi} \ar[d]^{e ^a}& \mathscr S_y \ar[r]^{e^{-u}}
 &\mathscr S_x\ar[d]_{e^a}\\
 \mathscr O_{\widetilde U} \ar[rr]^{\varphi}
 & & \mathscr O_{\widetilde U} 
 }
 \]
 commutative. Because $e^{-u}\circ \psi =\id \mbox{ mod } I$, 
 $\varphi =\id \mbox{ mod } I$. Therefore, there exists $q\in H^0(U,\Theta
 _U)\otimes I$ such that $\varphi = e^q$. Note that $e^q$ commutes with
 $e^a$ since the coefficients of $q$ are in $I$. This implies
 $\psi = e^u\circ e^{-a}\circ e^q\circ e^a= e^{u+q}$. 
 In other words, we have $\psi =e^r$ for $r=u+q\in KS_{\mathscr
 X}^0\otimes m_A$.
\end{proof}

\begin{remark}
 In the argument above, we can construct the lifting of the groupoid
 representation associated to $U\to \mathscr X$ 
 without appealing to Corollary \ref{dmcase}. Let $\widetilde
 R\rightrightarrows \widetilde U$ be the trivial
 lifting to $A$ of the groupoid $R\rightrightarrows U$ representing 
 $\mathscr X$ and $q_0=(s_0, t_0, e_0, m_0, i_0)$ the structural
 morphisms of $\widetilde R\rightrightarrows \widetilde U$. As we
 assumed $U$ Stein, every lifting of $R\rightrightarrows U$ to $A$ is
 given only by twisting $q_0$. The isomorphism $e^a :\mathscr
 S_x\overset{\sim}{\to}\mathscr O_{\widetilde U}$ appeared in the proof
 of Proposition \ref{NNmorph} induces an automorphism 
 \[
 \eta _a: 
 (\widetilde R, \mathscr O_{\widetilde R})
 \overset{s^*(e^a)}{\lto}
 (\widetilde R, s^{-1}\mathscr S_x) 
 = (\widetilde R, t^{-1}\mathscr S_x)
 \overset{t^*(e^{-a})}{\lto}
 (\widetilde R, \mathscr O_{\widetilde R}). 
 \]
 If we define $q_x=(s_x, t_x, e_x, m_x, i_x)$ by 
 \[
  s_x=s_0,\quad t_x=t_0\circ \eta _a,\quad
  e_x=e_0,\quad m_x=m_0\circ (p_2^*\eta _a^{-1}),\quad 
  i_x=i_0\circ \eta _a, 
 \]
 where $p_2$ is the projection 
 $\widetilde R\times _{\widetilde U} \widetilde R\to \widetilde R$ 
 to the second factor, 
 it is straightforward to check that $q_x$ satisfies 
 the axioms of a groupoid space and the lifting $q_x$ of the
 groupoid space actually corresponds to $\mathscr S_x\in \Def _{\mathscr
 O_{\mathscr X}}(A)$. In this way, we can prove Corollary 2.4 by hand for 
 a smooth separated DM stack $\mathscr X$. 
\end{remark}

\begin{theorem}\label{main}
 Let $\mathscr X$ be a smooth separated analytic Deligne-Mumford stack 
 and $KS_{\mathscr X}$ 
 be the associated Kodaira-Spencer algebra. Then, the morphism of functors 
 $\Gamma :\Def_{KS_{\mathscr X}} \to \Def _{\mathscr X}$
 in Theorem \ref{NNmorph-p} is actually an isomorphism in the analytic
 category.  
\end{theorem}

\begin{proof}
 To prove this theorem, we show that $\Gamma $ is \'etale. 
 According to Proposition \ref{criteria}, 
 this is equivalent to show that
 \begin{enumerate}[(i)]
  \item The `Dolbeault isomorphism' in Theorem \ref{dolbeault} 
  	on $H^1$ is actually the map induced
	by the morphism of functors $\Gamma :\Def _{KS_{\mathscr X}}\to 
	\Def _{\mathscr X}$. 
  \item The `Dolbeault isomorphism' in Theorem \ref{dolbeault} 
  	on $H^2$ satisfies the condition
	(ii) of Proposition \ref{criteria}.  
 \end{enumerate}
 For (i), the diagram chasing in \eqref{CechDolbeault} shows that 
 $x \in Z^1 (KS ^{\bullet}_{\mathscr X},\bar{\rd})$ corresponds to an
 element $g\in \Ker (\Gamma (U_1, \Theta)\overset{d}{\lto} 
 \Gamma (U_2, \Theta))$. As we consider over the ring of dual numbers
 $A=\mathbb C[\varepsilon]$, the condition $g\in \Ker d$ is equivalent to
 say that $(\id +g) :s^*\mathscr O_U\otimes A\to t^*\mathscr O_U\otimes A$
 satisfies the cocycle condition of descent and the sheaf on $\mathscr
 X$ given by descent with this twist is isomorphic to $\mathscr S_x$
 appeared in the proof of Theorem \ref{NNmorph-p}. 

 Now we prove (ii). 
 Let $0\to I\to A\to \bar A\to 0$ be a small extension. Let $x\in MC_{KS
 _{\mathscr X}(\bar A)}$ be a Maurer-Cartan solution on $\bar A$ and
 $\tilde x\in KS ^1_{\mathscr X}\otimes m_A$ an arbitrary lifting to $A$. 
 The obstruction to the existence of a lifting of $x$ in 
 $MC _{KS_{\mathscr X}}(A)$ is 
 \[
  h=\bar{\rd}\tilde x + \frac 12 [\tilde x,\tilde x]
  \in \Ker(KS^2_{\mathscr X}\otimes I \overset{\bar\rd}{\to}
  KS^3_{\mathscr X}\otimes I), 
 \]
 which does not depend on the choice of a lifting
 $\tilde x$. By a diagram chasing in \eqref{CechDolbeault} gives
 $\tau \in \Gamma (U_0,\mathscr A^{0,1}(\Theta))\otimes I$, 
 $\rho \in \Gamma (U_1,\mathscr A^{0,0}(\Theta))\otimes I$, and 
 $\omega \in \Ker (\Gamma (U_2,\Theta)\otimes I\to \Gamma (U_3,\Theta)\otimes I)$
 satisfying 
 \[
 \bar{\rd}\tau = h,\quad \bar{\rd}\rho = t^*\tau -s^*\tau,\quad
 \omega = p_2^*\rho - m^*\rho + p_1^*\rho,  
 \]
 where $p_1$, $m$, and $p_2$ are the projections $pr_{12}$, $pr_{13}$ 
 and $pr_{23}$ on 
 $U_2=U\times _{\mathscr X} U \times _{\mathscr X} U$, respectively. 
 If we put $\hat x=\tilde x-\tau$, we can check $\hat x\in
 MC_{KS_U}(A)$ (using the extension $A\to \bar A$ is small). 
 This means that there is a sheaf $\mathscr S_{\hat x}$ on $U$, 
 which is isomorphic to $\mathscr O_{\widetilde U}$. $\bar{\rd}\rho =
 t^*\tau -s^*\tau$ implies that $e^{\rho}:s^*\mathscr S_{\hat x}\to
 t^*\mathscr S_{\hat x}$ is an isomorphism (again using the smallness of
 the extension). The cocycle condition for descent is equivalent to the
 vanishing of the class $[\omega ] = [p_2^*\rho - m^*\rho + p_1^*\rho]$ 
 in the cohomology group $H^2(\mathscr X, \Theta _{\mathscr X})$. 
 This implies that $\omega$ gives the obstruction to lifting  
 $\mathscr S_{x}$ to $A$, thus we checked the condition (ii) of
 Proposition \ref{criteria}. 
\end{proof}


\end{document}